\documentclass{amsart}
\usepackage{amscd}
\usepackage{amssymb}
\usepackage{mathtools}

\usepackage[all, knot]{xy}
\xyoption{all}
\xyoption{arc}
\usepackage{hyperref}

\makeatletter
\newsavebox\myboxA
\newsavebox\myboxB
\newlength\mylenA
\newcommand*\xoverline[2][0.75]{%
    \sbox{\myboxA}{$\m@th#2$}%
    \setbox\myboxB\null
    \ht\myboxB=\ht\myboxA%
    \dp\myboxB=\dp\myboxA%
    \wd\myboxB=#1\wd\myboxA
    \sbox\myboxB{$\m@th\overline{\copy\myboxB}$}
    \setlength\mylenA{\the\wd\myboxA}
    \addtolength\mylenA{-\the\wd\myboxB}%
    \ifdim\wd\myboxB<\wd\myboxA%
       \rlap{\hskip 0.5\mylenA\usebox\myboxB}{\usebox\myboxA}%
    \else
        \hskip -0.5\mylenA\rlap{\usebox\myboxA}{\hskip 0.5\mylenA\usebox\myboxB}%
    \fi}

\newcommand*\xunderline[2][0.75]{%
    \sbox{\myboxA}{$\m@th#2$}%
    \setbox\myboxB\null
    \ht\myboxB=\ht\myboxA%
    \dp\myboxB=\dp\myboxA%
    \wd\myboxB=#1\wd\myboxA
    \sbox\myboxB{$\m@th\underline{\copy\myboxB}$}
    \setlength\mylenA{\the\wd\myboxA}
    \addtolength\mylenA{-\the\wd\myboxB}%
    \ifdim\wd\myboxB<\wd\myboxA%
       \rlap{\hskip 0.5\mylenA\usebox\myboxB}{\usebox\myboxA}%
    \else
        \hskip -0.5\mylenA\rlap{\usebox\myboxA}{\hskip 0.5\mylenA\usebox\myboxB}%
    \fi}
\makeatother

\newcommand{\overbar}[1]{\mkern 1.5mu\overline{\mkern-1.5mu#1\mkern-1.5mu}\mkern 1.5mu}

\def\Z{\mathbb{Z}}

\def\into{\hookrightarrow}
\def\onto{\twoheadrightarrow}

\numberwithin{equation}{subsection}

\theoremstyle{plain} 
\newtheorem{thm}[equation]{Theorem}
\newcounter{intro}
\newtheorem{introthm}[intro]{Theorem}
\newtheorem{introprop}[intro]{Proposition}

\newtheorem*{introthm*}{Theorem}
\newtheorem{cor}[equation]{Corollary}
\newtheorem{cor*}{Corollary}
\newtheorem{lem}[equation]{Lemma}
\newtheorem{lemma}[equation]{Lemma}
\newtheorem{prop}[equation]{Proposition}

\theoremstyle{definition}
\newtheorem{defn}[equation]{Definition}

\newtheorem{ex}[equation]{Example}

\newcommand{\numberseries}{\bfseries}   

\newlength{\thmtopspace}                
\newlength{\thmbotspace}                
\newlength{\thmheadspace}               
\newlength{\thmindent}                  

\renewcommand{\subparagraph}{\vspace*{\thmbotspace}}
\newtheoremstyle{fixed bf head,slanted body}
                {\thmtopspace}{\thmbotspace}{\slshape}
                {\thmindent}{\bfseries}{.}{\thmheadspace}
                {{\numberseries \thmnumber{#2\;}}\thmname{#1}\thmnote{
                    (#3)}}
\newtheoremstyle{it head,upright body}
                {\thmtopspace}{\thmbotspace}
                {\upshape}{\thmindent}{\upshape}{.}{\thmheadspace}
                {{\numberseries\thmnumber{#2\;}}
                {\itshape\thmnote{\negthickspace#3}}}
\theoremstyle{it head,upright body}

\theoremstyle{remark}

\newtheorem*{rems}{Remark}
\newtheorem*{remss}{Remarks}
\newtheorem*{quest}{Question}

\newtheorem*{ack}{Acknowledgements}

\newcommand{\Tor}[4]{\operatorname{Tor}^{#2}_{#1}(#3,#4)}

\newcommand{\Hom}[3]{\operatorname{Hom}_{#1}(#2,#3)}

\newcommand{\xra}[1]{\xrightarrow{#1}}
\newcommand{\xla}[1]{\xleftarrow{#1}}

\newcommand{\la}{\leftarrow}

\newcommand{\pd}{\operatorname{pd}}

\newcommand{\cone}[1]{\operatorname{cone}(#1)}

\newcommand{\fm}{\mathfrak{m}}
\newcommand{\fn}{\mathfrak{n}}

\newcommand{\cA}{\mathcal{A}}

\renewcommand{\t}[1]{\overbar{#1}}

\newcommand{\slte}[1]{\preccurlyeq #1}

\renewcommand{\Bar}[1]{\operatorname{Bar} #1}

\newcommand{\dhom}{d^{\operatorname{Hom}}}

\renewcommand{\o}[1]{\xoverline{#1}}

\newcommand{\p}[1]{{#1}_+}

\newcommand{\Ai}{$\text{A}_{\infty}$}
\newcommand{\A}[1]{$\text{A}_{#1}$}
\newcommand{\Aia}{$\text{A}_{\infty} \, \, A$}

\renewcommand{\u}[1]{\xunderline{#1}}

\newcommand{\LH}{Lef\'evre-Hasegawa }

\newcommand{\ponser}[2]{\operatorname{P}^{#1}_{#2}(t)}

\newcommand{\syz}[3][1]{\Omega^{#1}_{#2}(#3)}
\newcommand{\csyz}[3][1]{\operatorname{syz}_{#2}^{#1}(#3)}
\newcommand{\msyz}[1]{m^{\operatorname{syz}}_{#1}}

\newcommand{\cls}[1]{\operatorname{cls}(#1)}

\begin{document}
\title{Higher homotopies and Golod rings}
\author{Jesse Burke}
\address{Mathematics Department\\ 
UCLA\\ 
Los Angeles, CA\\ 90095-1555, USA}
\email{jburke@math.ucla.edu}

\begin{abstract}
We study the
homological algebra of an $R = Q/I$-module $M$ using \Ai-structures on
$Q$-projective resolutions of $R$ and $M$. We use these higher
homotopies to construct an $R$-projective bar resolution
of $M$, $Q$-projective
resolutions for all $R$-syzygies of $M$, and describe the differentials in the Avramov spectral
sequence for $M$. These techniques apply particularly well to Golod
modules over local rings. We characterize $R$-modules that are Golod
over $Q$ as those with minimal
\Ai-structures. This gives a construction of the minimal resolution of
every module over a Golod ring, and it also follows that if the inequality traditionally used to define Golod modules is an equality in the first
$\dim Q$+1 degrees, then the module is Golod, where no bound was
previously known. We also relate \Ai-structures on resolutions to Avramov's
obstructions to the existence of a dg-module structure. Along the way
we give new, shorter, proofs of several classical
results about Golod modules.
\end{abstract}

\maketitle
\newcommand{\env}[1]{{#1}^{\operatorname{ev}}}
\section{Introduction}
Let $R$ be an algebra over a commutative ring $Q$, and $M$ an
$R$-module. If $R$ and $M$ are projective as $Q$-modules, the bar
resolution gives a canonical $R$-projective resolution of $M$. To generalize to the case that $R$ or $M$ is not projective
over $Q$, we have to replace $R$, respectively $M$, by a $Q$-projective resolution
with an algebra, respectively module, structure. Indeed, let us assume that $R = Q/I$ is a cyclic $Q$-algebra,
a case of interest in commutative
algebra. In this case, Iyengar showed in \cite{MR1442152} that if $A
\xra{\simeq} R$ is a $Q$-projective resolution with a differential
graded (dg) algebra structure, and $G
\xra{\simeq} M$ is $Q$-projective resolution with a dg $A$-module
structure, then the bar resolution on
$A$ and $G$ gives an $R$-projective resolution of
$M$. Recall that a differential graded algebra/module, is an
algebra/module-over-that-algebra object in the category
of $Q$-chain complexes.

In this paper we generalize Iyengar's result by constructing an
$R$-projective resolution of $M$ using \Ai, or
``dg up-to-coherent-homotopy'', structures on $Q$-projective
resolutions. The distinction of \Ai-machinery is that every $Q$-projective
resolution of $R$ has an \Ai-algebra structure and every
$Q$-projective resolution of $M$ has an \Aia-module
structure, by \cite[\S 2.4]{Burke14}\footnote{This follows the
  classical spirit of \Ai-machinery, but because $A$ is not augmented and the ground ring is a
  not a field, it does not follow from arguments in the  literature.}. This is particularly useful in the study of local
rings, where we are interested in minimal resolutions. Avramov showed
that the minimal $Q$-free resolution of $R$ need not have a dg-algebra
structure, or if it does, the minimal $Q$-free resolution of $M$ need
not have a dg-module structure over it \cite{MR601460}. While $Q$-free
resolutions with dg structures
always exist, the dg bar resolution on a non-minimal
$Q$-resolution is non-minimal over $R$. Thus using \Ai-structures
removes a non-vanishing obstruction to a dg bar resolution being
minimal. 

When $Q$ is local Noetherian and $M$ is finitely generated, there is a
well-known upper bound on the size of a minimal $R$-free resolution of
$M$, called the Golod bound. The \Ai-bar resolution on minimal $Q$-free resolutions
realizes this bound. When the bound is
achieved $M$ is by definition a Golod module over $Q$. Summarizing the
previous two paragraphs, we have the following.
\begin{introthm}\label{introthm1}
Let $Q$ be a commutative ring, $R$ a cyclic $Q$-algebra, and $M$
an $R$-module. Let $A$ and $G$ be $Q$-projective
resolutions of $R$ and $M$, respectively. There is an $R$-projective bar resolution of
$M$ built from $A, G$ and \Ai-structures on these resolutions. If $Q$ is local Notherian and $M$ is finitely generated and $A, G$ are minimal, the bar resolution is
minimal if and only if $M$ is a
Golod module over $Q$.
\end{introthm}
In particular the minimality of the bar resolution does not depend on the \Ai-structures
chosen. 

The Golod bound first appeared in \cite{MR0138667}, and has been well-studied since. In particular it is known that
Golod modules occur with some frequency. To make this precise, we
recall some terminology. The map $Q \to R$ is a Golod
  morphism when the
residue field of $R$ is a Golod module over $Q$. The ring $R$ is a
Golod ring
if there is a Golod morphism $Q \to R$ with $Q$ regular local. Levin
showed that for a local ring $(Q, \fn)$, the map $Q \to Q/\fn^k$ is Golod for all $k \gg 0$
\cite[3.15]{MR0429868}. This was expanded upon by Herzog, Welker and
Yassemi, who showed that if $Q$ is regular and $I$ is any ideal, then $Q/I^k$ is Golod for all $k \gg 0$
\cite{1108.5862}, and more recently by Herzog and
Huneke \cite[Theorem 2.3]{MR3091800} who showed that if $Q$ is a standard graded polynomial algebra over a field of
characteristic zero and $I$ is any graded ideal, then $R = Q/I^k$ is a
Golod ring for all $k \geq 2$ (there is an obvious analogue of
Golodness for standard graded rings and graded modules over them). Finally, Lescot showed that the $\dim
Q$-th syzygy of every module over a Golod ring is a Golod module
\cite{MR1060830}. 

The interpretation of Golodness in terms of \Ai-structures in
Theorem \ref{introthm1} is quite useful.  The first new result is the following.
\begin{cor*}\label{introcor1}
If $Q$ is regular and the Golod bound is an equality in the first $\dim Q + 1$-degrees,
then $M$ is a Golod module.
\end{cor*}
There was no previous bound of this type known and the proof follows
surprisingly easily from the structure of the bar resolution. Indeed, an
\Ai-structure on $A$, respectively \Aia-module structure on $G$, is
given by a sequence of maps 
\[\widetilde m_n: A_{\geq 1}^{\otimes n} \to A_{\geq 1} \quad \text{respectively} \quad \widetilde
m_n^G: A_{\geq 1}^{\otimes n-1} \otimes G \to G\]
for all $n \geq 1$, where $\widetilde m_n$
has degree $2n - 1$ and $\widetilde m_n^G$ has degree $2n - 2$.  Since we
assume $Q$ is regular, $A$ and $G$ have length at most $\dim Q$. By degree considerations there are only finitely many
non-zero maps in any
\Ai-structure on $A$ or $G$, and by construction these all appear as
direct summands in the first $\dim Q +
1$-degrees of the bar resolution of $M$. Thus if the bar resolution is
minimal in these degrees, it must be minimal in all degrees. Corollary
\ref{introcor1} follows almost immediately.

Using the other direction of Theorem A, if $M$ is a Golod module, the bar resolution is a minimal
$R$-free resolution constructed from the finite data of $A, G$ and
their higher homotopy maps. Coupled with Lescot's result mentioned
above that every module over a Golod ring has a syzygy that is a Golod
module, we have the following.
\begin{cor*}\label{introcor2}
If $R$ is a Golod ring, there is a finite construction of the minimal
free resolution of every finitely generated $R$-module.
\end{cor*}
This adds a large new class of rings for which an explicit
construction of minimal resolutions exist. And let us emphasize this is a construction. Finding the
non-zero maps $\widetilde m_n$ and $\widetilde m_n^G$ in \Ai-structures
on $A$ and $G$ can be implemented on a
computer, using e.g.\ \cite{M2} (although it is currently only feasible for
small examples). We also give a new proof
of Lescot's result in \ref{syzs_over_golod_ring_are_golod}, using techniques developed earlier in
the paper, so Corollary \ref{introcor2} is self-contained. 

In addition to the the bar resolution, we introduce two other constructions that use \Ai-structures to
study $R$-modules. To describe the first, we let $\syz R M$ be the
first syzygy of $M$ over $R$.
\begin{introprop}\label{introthm2}
Let $Q$ be a commutative ring, $R$ a cyclic $Q$-algebra, and $M$
an $R$-module. Let $A$ and $G$ be $Q$-projective
resolutions of $R$ and $M$ with \Ai algebra and \Aia-module
structures, respectively. There is a complex $\csyz R G$, built from
$A, G$ and \Ai-structures on these complexes, with a canonical
\Aia-module structure and a quasi-isomorphism
\[ \csyz R G \xra{\simeq} \syz R M.\]
Thus $\csyz R G$ is a $Q$-projective
resolution of the
  first $R$-syzygy of $M$.
\end{introprop}
This is a special case of a well-known procedure for calculating a
$Q$-projective resolution of an $R$-syzygy of $M$ using
a mapping cone. The novelty here is that we do not have to choose a
lift of the multiplication map $R \otimes M \to M$ to $Q$-free
resolutions: it is contained in the
\Ai-structure. And since the new $Q$-resolution has an \Aia-module
structure, one can iterate and construct canonical $Q$-projective resolutions of all $R$-syzygies of
$M$. If $Q$ is local and $A, G$ are minimal, these new resolutions are
not necessarily
minimal, but are for Golod modules. Thus we have a closed
description of the
minimal $Q$-free resolution of every $R$-syzygy of a Golod module, in
terms of \Ai-structures on $A$ and $G$. 

The second construction uses \Ai-structures to describe the
differentials of the Avramov spectral sequence. This spectral sequence, introduced in \cite{MR601460}
using an Eilenberg-Moore type construction, transfers information from
$Q$ to $R$, the opposite direction of the standard change of rings spectral
sequence. Specifically, we have the following.
\begin{introthm}\label{introthm3}
 Let $Q$ be local with residue field $k$ and $A$ and $G$ minimal $Q$-free
 resolutions of $R$ and $M$, respectively. \Ai-structures on $A$
 and $G$ describe the differentials in the
  Avramov spectral sequence for $M$:
\[ E^2_{p,q} = \left ( \Tor p {\Tor * Q R k} {\Tor * Q M k} k \right)_q \Rightarrow \Tor {p+q} R M k.\] 
\end{introthm}

While the precise description of the differentials is technical, as a corollary we can describe the relation of \Ai-structures to
the obstructions constructed in \cite{MR601460}. To state this, assume that
$A$ is a dg-algebra. Avramov
constructed a group which vanishes if $G$ has a dg $A$-module
structure. This is equivalent to the existence of an \Aia-module structure $(\widetilde m_n^G)$ with $\widetilde m_n^G = 0$ for
all $n \geq 3$. Thus if the
obstruction is nonzero, then $\widetilde m_n^G$ is nonzero for some $n \geq
3$ in any 
\Aia-module structure on $G$. But in fact, using Theorem \ref{introthm3} we
show more is true.
\begin{cor*}
If Avramov's obstruction is non-zero, then in any \Aia-module
structure $(\widetilde m_n^G)$ on $G$, $\widetilde m_n^G$ contains a unit for some $n \geq 3$.
\end{cor*}
Thus these obstructions only detect the existence of a dg-structure modulo the maximal ideal
of $R$. This is motivated by a strategy sketched in \cite{MR601460} to prove the Buchsbaum-Eisenbud
rank conjecture \cite{MR0453723}, where the key step is finding an
obstruction theory that detects precisely when $G$ has a dg-module structure. It would be enough to do this in the case $A$ is a Koszul complex.

The three tools described above (the lettered results) work well in conjunction. We
use them here to give new, short, proofs of the following four classical results on Golod
modules. 
\begin{cor*}
Let $(Q, \fn, k)$ be a regular local ring, $R =  Q/I$ with $I
\subseteq \fn^2$, and $M$ a finitely generated $R$-module.
\begin{enumerate}
\item If $M$ is a Golod module, then its first syzygy $\syz R
  M$ is also Golod.
\item If $M$ is a Golod module, then $R$ is a Golod ring.
\item If $R$ is Golod, then $\syz [d+1] R M$ is a
  Golod module, where $d = \dim Q$.
\item $M$ is Golod if and only if the change of rings maps
\[  \Tor * Q M k \to
\Tor * R M k \quad \mu: \Tor * Q {\syz R M} k \to \Tor * R {\syz R
  M}k\]
are injective.
\end{enumerate}
\end{cor*}
The first is due to Levin \cite{MR814185}, the second, independently, to
Lescot \cite{MR1060830} and Levin \cite{MR814185}, the third
Lescot \cite{MR1060830}, and the fourth Levin \cite{MR0429868}. Let us quickly sketch
our proofs. It follows from Theorem \ref{introthm1} that if $M$ is
Golod, then any \Ai-structures $\widetilde m, \widetilde m^G$ are minimal. The
\Ai-structure maps on the $Q$-free resolution $\csyz R G$ of $\syz R
M$ given in Proposition \ref{introthm2} are built from $\widetilde m$
and $\widetilde m^G$,
thus are also minimal, and so $\syz R M$ is Golod, again by
Theorem \ref{introthm1}. Using the computations of Theorem
\ref{introthm3}, we show in
Theorem \ref{golod_ring_thm} that $R$ is Golod if and only if
$\widetilde m$ is
minimal (i.e.\ this implies that the
\Aia-module structure on the resolution of $k$ is minimal). Coupled
with Theorem \ref{introthm1}, this makes it clear that if $M$ is a Golod module, then $R$ is a Golod ring. To prove the
third result, we show that for $n
> \dim Q$, the $n$th iterate of the syzygy resolution of Proposition \ref{introthm2}
has an \Aia-module structure with higher maps constructed
entirely from $\widetilde m$ (i.e.\ we have iterated
the maps of $\widetilde m^G$ into the differential). So if $R$ is Golod,
then $\widetilde m$ is minimal, and this implies the higher
homotopies for a
resolution of
$\syz [n] R M$ are also minimal when $n > \dim Q$, so the syzygy
module is Golod. The proof of the fourth
result uses the fact that the change of
rings maps are edge maps in the Avramov spectral sequence.

To end the introduction, let us discuss the context and future
directions of this work. Dg-techniques have been used been widely in homological commutative
algebra, see 
\cite{MR1648664} and its references, but \Ai-machinery
has been used sparingly, if at all. One probable reason is that
unless $Q \to R$ is the identity map, the resolution $A$
is not augmented over $Q$. And until recently, most parts of the \Ai-machinery
required an augmentation. Positselski showed in \cite{MR2830562} how to
compensate for a lack of augmentation by including a curvature
term on the bar construction. It is not made explicit below, but this is the
key to using \Ai-structures as we do. There are
further details on this setup in \cite[\S 2]{Burke14}.

The results of the
paper generalize in a straightforward way to the case $Q \to R$ is module
finite\footnote{This assumption is necessary to discuss the
difference in the size of minimal $Q$ and $R$ free resolutions of an
$R$-module $M$ in the way we do here, in particular for a Golod bound
to exist.} and the basic setup of the bar resolution
generalizes to the case $R$ is an arbitrary \Ai-algebra over $Q$. We assume that $R$ is cyclic to simplify the definition of
\Ai-algebra on a $Q$-projective resolution of $R$; in general the curvature term is non-zero in higher
degrees, and this complicates notation, especially in definitions. There is a very interesting generalization in a
 different direction. Here we are implicitly using the
twisting cochain $\Bar A \to A$ (see e.g.\ \cite{MR0436178} and
\cite{MR0365571} for classical references on dg-bar constructions and
twisting cochains, and
\cite{Burke14} for generalizations of the definitions and basic
results to \Ai-objects and curvature on the
coalgebra side) called the universal twisting cochain. When $R$ is Golod, our results
show that the universal twisting cochain is the smallest \emph{acyclic}
twisting cochain, i.e.\ one that preserves all the homological data
of $R$. The main goal of \cite{Burke14} is to give tools to find
acyclic twisting cochains $C \to A$, with $C$ smaller than $\Bar A$. For
instance, when $R$ is a complete intersection and $A$ is the Koszul
complex on defining equations of $R$, there is a
twisting cochain $C \to A$ with $C$ the divided powers on
$A_1[1]$, and this explains the polynomial behavior of infinite
resolutions over complete intersections. We view this paper as the Golod chapter in a
series, with further details on complete intersections to follow
soon. The next natural class of rings
to consider are codimension 3
Gorenstein rings (every codimension 2 ring is complete
intersection or Golod). We give a running example of a specific codimension 3
Gorenstein ring that we hope may hint at the possibilities of the machinery.

\begin{ack}
Thank you to Lucho Avramov, David Eisenbud, Srikanth Iyengar,
Frank Moore, and Mark Walker for helpful
conversations and questions, and Jim Stasheff for comments on the manuscript.
\end{ack}

\section{Conventions}
By $Q$ we denote a commutative ring and we always work over this ring. In particular, all unmarked tensor and hom
functors are defined over $Q$. For graded
  modules $M,N$, let $\Hom {} M N$ be the graded module with degree
  $i$ component 
  $\prod_{j \in \Z} \Hom {} {M_j} {N_{j + i}}$. We use
  homological indexing for complexes, so differentials lower
  degree. If $M$ and
  $N$ are complexes, $\Hom {} M N$ is a complex with differential
  $\dhom(f) = f d_{M} - (-1)^{|f|} d_{N} f$. A map of chain complexes
  is a cycle in $\Hom {} M N$ and a morphism of complexes is a
  degree zero
  map of chain complexes. If $M$ is a complex, $M[1]$ is the complex with $M[1]_n = M_{n
    -1}$ and $d_{M[1]} = -d_M$. We set
\[ s: M \to M[1]\]
to be the degree 1 map with $s(m) = m$.
When evaluating tensor products of homogeneous
  maps we use the sign convention $(f \otimes g)(x \otimes y)
  = (-1)^{|g||x|}f(x) \otimes g(y)$. If $M, N$ are complexes,
  $M \otimes N$ is a complex with differential $d_M
  \otimes 1 + 1 \otimes d_N$. All elements of graded objects are
  assumed to be homogeneous.

 If $f: M \to N$ is a morphism of complexes,
  $\cone f$ is the complex with underlying module $M[1] \oplus N$ and
  differential
\[ d^{\cone f} = \left ( \begin{matrix} d^{M[1]}
            & 0 \\ f s^{-1} & d^N\end{matrix} \right ): M[1] \oplus N \to M[1] \oplus N. \]
If $M$ is a complex and $x$ a cycle in $M$, $\cls x$ is
the class of $x$ in $H_*(M)$.

Throughout $R$ is a cyclic $Q$-algebra, $M$
an $R$-module, and
\[A \xra{\simeq} R \quad  \enspace G \xra{\simeq} M\]
$Q$-projective resolutions. We assume that $A_0 = Q$ and set $\p
A = A_{\geq 1}$. We do not assume that $R$ or $M$ has finite
projective dimension over $Q$.

\section{A-infinity structures on resolutions}
\numberwithin{equation}{section}
In this section we recall the definition of \Ai-algebra and
module, give a short proof that such structures exist as claimed in the
introduction, and use them to construct an $R$-projective bar resolution of $M$. For an
introduction to \Ai-algebras, see \cite{MR1854636}. We
use some definitions and very basic results from \cite{Burke14}, but
emphasize that none of the heavy machinery of that paper is used here.

We denote the differential of the
complex $A$ as $d^A$. In the following definition we view $d_1^A$ as a
degree --2 map of chain complexes $\p A[1] \to A_0$ and make the identifications $A_0 \otimes \p A[1] = \p A[1] =\p A[1] \otimes A_0$.
\begin{defn}\label{defn_ai_alg}
  An \emph{\Ai-algebra structure on $A$} is a
  sequence of degree --1 maps $m = (m_n) = (m_n)_{n \geq 1}$, with $m_n:
  \p A[1]^{\otimes n} \to \p A[1]$, satisfying:
  $$m_1 = d^{A[1]}|_{\p A[1]},$$
  $$m_1 m_2 + m_2( m_1 \otimes 1 + 1 \otimes m_1) = d_1^A \otimes
    1 - 1 \otimes d_1^A,$$
$$\text{ and } \quad \sum_{i= 1}^{n} \sum_{j = 0}^{n - i} m_{n - i + 1} (
    1^{\otimes j} \otimes m_{i} \otimes 1^{\otimes n - i - j} ) = 0
    \text{ for all } n \geq 3.$$
If $Q$ is graded, $I$ is homogeneous and $A$ is a graded projective
resolution, $m$ is a \emph{graded \Ai-algebra structure} if each
$m_n$ preserves internal degrees.
\end{defn}

\begin{remss}We note the following:
  \begin{enumerate}
  \item Using $\p A[1]$ instead of $\p A$ removes signs from the
    definition.
\item Using $\p A$
    instead of $A$ ensures that our \Ai-algebras are \emph{strictly
      unital}. Indeed, let $1_A$ be an element of $A_0$ that maps to
    $1 \in R$. If $(m_n)$ is an \Ai-structure on $A$,
    then one can uniquely extend each $m_n$ to a map $m_n': A[1]^{\otimes n}
    \to A[1]$ such that $m_1' = d^{A[1]}$, $m_2'( a \otimes 1_A )
    = (-1)^{|a|+1} a$ and
    $m_2'( 1_A \otimes a) = -a$ for all $a \in A$, and $m_n'$ vanishes on any term
    containing $1_A$, for $n \geq 3$.

\item The map $d_1^A \otimes 1 - 1 \otimes d_1^A$ is a \emph{curvature term},
    introduced by Positselski,
    accounting for the fact that $A$ is not augmented. See \cite[\S
    7]{MR2830562} and \cite[2.2.8]{Burke14} for further details.
\item To match up with notation from the introduction, set $\widetilde m_n = s m_n
  (s^{-1})^{\otimes n}$.
  \end{enumerate}
\end{remss}

\begin{ex}\label{dga_is_aia}
Let $i: \p A \into A$ and $p:
A \to \p A$ be the canonical maps. If $A$ has a
dg-algebra structure with multiplication $\mu$, then
\[ m_2 = sp\mu (is^{-1} \otimes is^{-1}): \p A[1] \otimes \p A[1] \to
\p A[1]\]
and $m_n = 0$ for $n \geq 3$, is an \Ai-algebra structure on
$A$. 
\end{ex}

\begin{defn}\label{defn_ai_mod}
Let $(m_n)$ be an \Ai-algebra structure on $A$. An \emph{\Aia-module
    structure on $G$} is a
  sequence of degree --1 maps $m^G = (m_n^G) = (m_n^G)_{n \geq 1}$, with $m_n^G: \p
  A[1]^{\otimes n - 1} \otimes G \to G$, satisfying:
  \begin{displaymath}
 m_1^G = d^G,
\end{displaymath}
\begin{displaymath}
m_1^G m_2^G + m_2^G( m_1 \otimes 1 + 1 \otimes m_1^G) = d_1^A
      \otimes 1,
    \end{displaymath}
\begin{displaymath}
\sum_{i= 1}^{n} \sum_{j = 0}^{n - i} m_{n - i + 1}^G (
      1^{\otimes j} \otimes m_{i} \otimes 1^{\otimes n - i - j} ) = 0
     \text{ for all } n \geq 3,
  \end{displaymath}
where the $m_i$ in the last line is $m_i^G$ when $j = n -
      i$. If $Q, R, M$ are graded and $A, G$ are graded resolutions,
      $m^G$ is a \emph{graded
  \Aia-module structure} if each $m_n^G$ preserves internal degrees.
\end{defn}

\begin{ex}
If $A$ is a dg-algebra and $G$ is a dg $A$-module with structure map $\mu^G: A
 \otimes G \to G,$ then
\[ m^G_2 = \mu^G (is^{-1} \otimes 1): \p A[1] \otimes G \to G\]
and $m^G_n = 0$ for $n \geq 3$
is an \Aia-module structure on $G$.
\end{ex}

\begin{ex}\label{hyp_ex}
Let $R = Q/(f)$ with $f$ a non-zero divisor and set $A
= 0 \to Q \xra{f} Q \to 0$. Then $A$ is
a $Q$-projective resolution of $R$ and $m_n = 0$ for all $n$ is an
\Ai-algebra structure on $A$. If $m^G$ is an
\Aia-module structure on
$G$, identify $m_n^G$ with a degree $-1$ map $G[2n-2] \to G$, e.g.\ a
degree $2n-3$ map $G \to G$. Then $m_2^G$ is a homotopy
for multiplication by $f$ and for $n \geq 3$,
\[ \sum_{i = 1}^n m^G_{n-i+1} m^G_i = 0.\]
Thus an \Aia-module structure on $G$ is the
same as a system of higher homotopies in the sense of \cite{MR0241411}
and \cite{Ei80}, where $\sigma_i = m^G_{i+1}$.
\end{ex}

It follows from \cite[Theorem 2.4.5]{Burke14} that there always exists an
\Ai-algebra structure on $A$ and an \Aia-module structure on $G$. We give a shorter proof 
below for the special case we are working with in this paper. This proof also gives an
algorithm for constructing \Ai-structures by computer, and shows that if
$R, M, A$ and $G$ are graded, then graded \Ai-structures on $A$ and
$G$ exist.
\begin{prop}\label{ai_exist}
, $R$ a cyclic $Q$-algebra, $M$ an $R$-module, and 
\[A \xra{\simeq} R \quad G \xra{\simeq} M\]
$Q$-projective resolutions. There exists an \Ai-algebra structure on $A$ and an \Aia-module
 structure on $G$.
\end{prop}

\begin{proof}
  We construct the sequence $(m_n)$ inductively. Set $m_1 = d^{A[1]}|_{\p A[1]}$. To construct $m_2$,
 consider the degree --2 map
\[ \alpha = d_1^A \otimes 1 - 1 \otimes d_1^A: \p A[1] \otimes \p A[1] \to \p
A[1],\]
where $d_1^A$ is viewed as a degree --2 map $\p A[1] \to A_0$ and we have identified $A_0 \otimes \p A[1] \cong \p A[1] \cong \p
A[1] \otimes A_0$. Since $d_1^A$ is a map of chain complexes, $d_1^A \otimes 1 - 1 \otimes
d_1^A$ is also. The map $\alpha$ has the form:
\[ \xymatrix{ 0 & \ar[l] A_1 \otimes A_1 \ar[dd]_{d_1^A \otimes 1 - 1
    \otimes d_1^A} & \ar[l] \substack{A_1 \otimes A_2\\A_2 \otimes A_1}
  \ar[dd]_{\tiny
    \begin{pmatrix}
      d_1^A \otimes 1 \\
-1 \otimes d_1^A
    \end{pmatrix}}
& \ar[dd] \ar[l]
\substack{A_1^{\otimes 3} \\ A_1 \otimes A_3\\A_2 \otimes
  A_2 \\ A_3 \otimes A_1} \ar[dd]_{\tiny
  \begin{pmatrix}
    0 \\ d_1^A \otimes 1 \\ 0 \\ -1 \otimes d_1^A
  \end{pmatrix}}
& \ar[l] \ldots \\
& & & & \\
0 & A_1 \ar[l] & A_2 \ar[l] & A_3 \ar[l] & \ar[l] \ldots}\]
The first
non-zero component induces the zero map in homology and thus by the classical lifting result
\cite[V.1.1]{MR0077480} is nullhomotopic. Take any homotopy as $m_2$. Given $m_1, \ldots, m_{n-1}$, for
$n \geq 3$, we use the obstruction
lemma of \LH \cite[Appendix B]{LH}, modified for strictly unital
algebras in \cite[\S 2.3]{Burke14}, to construct $m_n$. By \cite[2.3.3]{Burke14},
\[r(m|_{n-1}) := \sum_{i= 2}^{n-1} \sum_{j = 0}^{n - i} m_{n - i + 1} (
    1^{\otimes j} \otimes m_{i} \otimes 1^{\otimes n - i - j} )  \]
is a degree --1 cycle in the complex $\Hom {} {\p A[1]^{\otimes n}} {\p
  A[1]}$, where
$\p A[1]^{\otimes n}$ has the usual tensor differential, and a map $m_n$ satisfies the
\Ai-relations if and only if 
\[\dhom(m_n) + r(m|_{n-1}) =
0.\] 
But for $n \geq 3$,
\[ H_{-1} \Hom {} {\p A[1]^{\otimes n}} {\p A[1]} \cong H_{n-2} \Hom {} {\p
  A^{\otimes n}} I  = 0\]
since $\Hom {} {\p
  A^{\otimes n}} I$ is concentrated in non-positive degrees.
Thus $r(m|_{n-1})$ is a boundary. Pick any
map in the preimage for $m_n$.

To put an \Aia-module structures on $G$, set $m_1^G = d^G$ and let $m_2^G$
be a homotopy for the null-homotopic map $d_1^A \otimes 1: \p
A[1] \otimes G \to G$. Then one proceeds analogously as for $A$, using the analogue of the obstruction
theory for modules given in \cite[\S 3.3]{Burke14}.
\end{proof}

By \cite[Theorem 2.4.5]{Burke14} the \Ai-structures of Proposition
\ref{ai_exist} are unique up to homotopy. We will need the following
weaker results, whose proofs are clear.
\begin{cor}
Using the notation of \ref{ai_exist}, the following are equivalent.
\begin{enumerate}
\item there exists an \Ai-algebra structure $(m_n)$ on $A$ with $m_3 =
  0$;
\item there exists an \Ai-algebra structure $(m_n)$ on $A$ with $m_n =
  0$ for all $n \geq 3$;
\item $A$ has a dg-algebra structure.
\end{enumerate}
\end{cor}

\begin{cor}\label{obs_cor}
Assume that $A$ has a dg-algebra structure. The following are equivalent.
\begin{enumerate}
\item there exists an \Aia-module structure $(m_n^G)$ on $G$ with $m_3^G =
  0$;
\item there exists an \Aia-module structure $(m_n^G)$ on $G$ with $m_n^G =
  0$ for all $n \geq 3$;
\item $G$ has a dg $A$-module structure.
\end{enumerate}
\end{cor}

\begin{quest}
It would be interesting to compute \Ai-structures in special
cases, e.g.\ in Example \ref{ex_of_hhs} below, or for the resolutions
Srinivasan shows in \cite{MR1152904} have no dg-algebra structure. More ambitiously, if $A$ or $G$ has a combinatorial structure, e.g.\ if $A$ is a
cellular resolution \cite{MR1647559}, can one give a corresponding
combinatorial description of the \Ai-structures?

In another direction, we ask what
restrictions do \Ai-structures place on the Boij-Soederberg
decompositions \cite{MR2505303} of $A$ and $G$, and vice versa?
\end{quest}

\begin{ex}
If $A$ and $G$ have length $2$, then by degree considerations, $m_n = 0
= m_n^G$ for
all $n \geq 2$, i.e\ $A$ is a dg-algebra and $G$ is a dg $A$-module.
If $A$ has length 3, then by \cite[1.3]{MR0453723}, $A$ is a dg-algebra.
\end{ex}

This next example shows that $G$ can have non-zero $m_3^G$ 
when $A, G$ both have length 3. I do not know if it is possible
to put a dg $A$-module structure on this complex.
\begin{ex}\label{codim_3_ex}
Let $k$ be a field and $Q = k[[x,y,z]]$. Let $I = (x^2, -yz, xy+z^2,
-xz, y^2)$ be the ideal generated by the submaximal
Pfaffians of the alternating matrix
\[ \phi =
\left ( \begin{array}{rrrrr}
0 & y & 0 & 0 & z \\
-y & 0 & x & z & 0\\
0 & -x & 0 & y & 0\\
0 & -z & -y & 0 & x\\
-z & 0 & 0 & -x & 0
\end{array} \right ),\]
and set $R = Q/I$. By \cite{MR0453723}, $R$ is a codimension 3 Gorenstein ring and
the minimal $Q$-free resolution of $R$ is
\[ A: \quad 0 \la Q \xla{ d_1} Q^5 \xla{\phi} Q^5 \xla{d_1^{\operatorname{Tr}}} Q \la 0\]
with $d_1 =
\left ( \begin{array}{rrrrr}
 x^2 & -yz & xy + z^2 & -xz & y^2
\end{array} \right )$. Let $(a_1, \ldots, a_5)$, $(b_1,
\ldots, b_5)$ and $(c_1)$ be bases for $A_1, A_2$ and $A_3$, respectively. By
\cite[4.1]{MR0453723}, $A$ is a graded-commutative dg-algebra with
multiplication table
\[a_1 a_2 = x b_3 + yb_5 \quad a_1 a_3 = -xb_2 - z b_5 \quad a_1 a_4 = x b_5
\quad a_1 a_5 = -yb_2 + zb_3 - xb_4\]
\[ a_2 a_3 = xb_1 - zb_4 \quad a_2 a_4 = zb_3 \quad a_2 a_5 = y b_1
\quad a_3 a_4 = -zb_2 + y b_5\]
\[a_3 a_5 = -zb_1 - yb_4 \quad a_4 a_5 = xb_1 + yb_3 \quad a_i b_j =
\delta_{ij} c_1 \text{ for } 1 \leq i, j \leq 5.\]
(See also \cite[2.1.3]{MR1648664} for details on this construction.) We
consider the dg-algebra $A$ as an \Ai-algebra, as in Example \ref{dga_is_aia}.

Let $K \xra{\simeq} k$ be the Koszul complex over $Q$. Since $K$ is the $Q$-free
resolution of the $R$-module $k$, it has an \Aia-module structure. To
construct one explicitly, we consider the following map of complexes.
\[ \xymatrix{
A: & 0 & \ar[l] Q \ar[d]_= & Q^5 \ar[d]_{\alpha_1} \ar[l] & Q^5
\ar[d]_{\alpha_2} \ar[l] & Q \ar[l] \ar[d]_{\alpha_3} & \ar[l] 0\\
K: & 0 & Q \ar[l] & \ar[l]^{d_1^K} Q^3 & Q^3 \ar[l]^{d_2^K}
& Q \ar[l]^{d_3^K} & \ar[l] 0
}\]
with
\[ \alpha_1 =
\begin{pmatrix*}[r]
  x & 0 & 0 & 0 & 0\\
0 & 0 & x & 0 & y\\
0 & -y & z & -x &0
\end{pmatrix*}
\quad
\alpha_2 =
\begin{pmatrix*}[r]
 0 & -x & 0 & 0 & 0\\
0 & 0 & 0 & 0 & -x\\
y & 0 & 0 & 0 &0
\end{pmatrix*}
\quad
\alpha_3 =
\begin{pmatrix*}[r]
  xy
\end{pmatrix*},
\]
where
\[ d_1^K =
\begin{pmatrix*}[r]
 x & y & z 
\end{pmatrix*}
\quad
d_2^K = \begin{pmatrix*}[r]
 -y & -z & 0\\
x& 0 & -z\\
0 & x & y 
\end{pmatrix*}
\quad
d_3^K =
\begin{pmatrix*}[r]
 z \\ -y \\ x 
\end{pmatrix*}.
\]

Set $m_1^K = d^K$ and
\[m_2^K = \mu^K ( -\alpha i s^{-1} \otimes 1): \p A[1] \otimes K \to K,\]
where $\mu^K: K \otimes K \to
K$ is multiplication. Using that $\alpha$ is a map of complexes and
that $m_1 = -(d^A_{\geq 2} - d_1^A)$, one checks that $m_1^K m_2^K + m_2^K( m_1
\otimes 1 + 1 \otimes m_1^K) = d_1^A \otimes 1$.

We now calculate $m_2^K( m_2
\otimes 1 + 1 \otimes m_2^K)$. For degree reasons, the only possible non-zero component is $A_1 \otimes A_1 \otimes K_0 \to
K_2$. On this component, we have
\[m_2^K( m_2
\otimes 1 + 1 \otimes m_2^K)  = -\alpha_2\mu^A + \mu^K(
\alpha_1 \otimes \alpha_1).\]
The matrix of this map with respect to the basis $a_i \otimes a_j$ of $A_1 \otimes
A_1$, ordered so that the first two elements are $a_3 \otimes a_4$ and $a_4 \otimes a_3$, is
\[\arraycolsep=2.5pt
\left ( \begin{array}{rrrrr}
-xy & xy & 0 & 0 \ldots\\
xy & -xy & 0 & 0 \ldots\\
-x^2 & x^2& 0 & 0 \ldots\\
\end{array} \right ).
\]
Thus, defining $m_3^K: \p A[1]^{\otimes 2} \otimes K \to K$ on the
component $A_1 \otimes A_1 \otimes K_0 \to K_3$ by
\[m_3^K = \left ( \begin{array}{rrrrr}
x & -x & 0 & 0 \ldots
\end{array} \right ),\]
and zero elsewhere, gives
\[m_3^K( m_1 \otimes 1^{\otimes 2} + 1 \otimes m_1 \otimes 1 + 1^{
\otimes 2} \otimes m_1^K) + m_1^K m_3^K =  m_1^K m_3^K = -m_2^K( m_2
\otimes 1 + 1 \otimes m_2^K),\] and so $m^K$
is an \Aia-module structure on $K$.
\end{ex}

\begin{ex}\label{ex_of_hhs}
Let $Q = k[[x,y,z,w]]$, $R = Q/(x^2, xy, yz, zw,
w^2)$ and $A$ be the minimal $Q$-free resolution of $R$. Then $\pd_Q
R = 4$, and by \cite[2.3.1]{MR1648664} there is no dg-algebra
structure on $A$. Thus every \Ai-algebra
structure on $A$ has nonzero $m_3$.
\end{ex}
The proof of the above uses the obstuctions
of \cite{MR601460}; we relate these to \Ai-structures in Section \ref{obstruction_sec}.

\begin{defn}
  Let $\Bar A$ be the module $\bigoplus_{n \geq 0} \p
  A[1]^{\otimes n}$ and write $[x_1 | \ldots |
x_n]$ for the element $s(x_1) \otimes \ldots \otimes s(x_n) \in \p
A[1]^{\otimes n} \subseteq \Bar A$. The \emph{bar resolution} on $A$ and $G$ has underlying module
  \[ R \otimes \Bar A \otimes G = \bigoplus_{n \geq 0} R \otimes \p
  A[1]^{\otimes n} \otimes G\]
and differential
\[
      d( { [x_1| \ldots | x_n] \otimes g} ) =  R \otimes \left (\sum_{i
        = 1}^n \sum_{j = 0}^{i-1}
      {[\u x_1|\ldots |\u x_j|m_i( [x_{j+1}| \ldots | x_{j +
        i}])|\ldots|x_n]
      \otimes g } \right .\]
\[     \left. + \sum_{i = 0}^n { [\u x_1
      | \ldots |\u x_{n-i+1}] \otimes m_{i}^G( [x_{n-i+2}| \ldots |x_n]
      \otimes g) } \right )
\]
where $\u x =
  (-1)^{|x|+1}x.$
Let $\epsilon_G: G \to M$ be the
augmentation of the resolution. Define a map
\[ \epsilon_{\Bar {}}: R \otimes \Bar A \otimes G \to M\]
by $\epsilon_{\Bar {}}( { [\, \,  ] \otimes g } ) = \epsilon_G(g)$ and
$\epsilon_{\Bar {}}( { [x_1 | \ldots | x_n] \otimes g} ) = 0$ for $n \geq
1$.
\end{defn}

\begin{rems}
The module $\Bar A = \bigoplus_{n \geq 0} \p
  A[1]^{\otimes n}$ has a graded coalgebra structure (the tensor
  coalgebra) and a coderivation induced by the
  \Ai-algebra structure on $A$. Together with the $Q$-linear map $\Bar A
  \onto \p{A_1} \xra{d_1} Q = A_0$, they form a curved dg-coalgebra called the
  \emph{bar construction} or bar coalgebra of $A$. 
\end{rems}

\begin{thm}\label{proj_resln_thm}
 Let $Q$ be a commutative ring, $R$ a cyclic $Q$-algebra, $M$ an $R$-module, and 
\[A \xra{\simeq} R \quad G \xra{\simeq} M\]
$Q$-projective resolutions with \Ai-algebra and \Aia-module
structures, respectively. The bar resolution is a complex of
projective $R$-modules and the map
\[ \epsilon_{\Bar {}}: R \otimes \Bar A \otimes G \to M\]
makes it an $R$-projective resolution of $M$.
\end{thm}

When $A$ is a dg-algebra and $G$
is a dg $A$-module, this is
\cite[1.4]{MR1442152}.
\begin{proof}
The results of \cite[\S 7]{Burke14} show that $R \otimes \Bar A
\otimes G$ is a complex with homology $M$. One can also check
directly, using the
definition of \Ai algebra and module, that it is a complex. And to show it is exact, one can use
an analogous proof as in \cite[1.4]{MR1442152}, by filtering the complex by
the number of bars and considering the resulting spectral sequence.
\end{proof}

\begin{ex}
Continuing Example \ref{hyp_ex}, $R \otimes
\Bar A \otimes G$ is $R \otimes -$ applied to:
\[ 0 \la G_0 \xla{m_1^G} G_1 \xla{\tiny
  \begin{pmatrix}
    m_2^G & m_1^G 
  \end{pmatrix}} \begin{array}{c} G_0 \\ G_2\end{array}\xla{\tiny
    \begin{pmatrix}
      m_2^G & m_1^G\\  m_1^G & 0
    \end{pmatrix}}
\begin{array}{c} G_1 \\ G_3\end{array}
\xla{\tiny
  \begin{pmatrix}
    m_3^G & m_2^G & m_1^G\\
m_2^G & m_1^G & 0
  \end{pmatrix}
}
\begin{array}{c}
  G_0 \\ G_2 \\ G_4
\end{array}
\la \ldots
\]
This is the resolution constructed from a system of higher homotopies in \cite[3.1]{MR0241411} and \cite[\S 7]{Ei80}, for hypersurfaces.
\end{ex}

\begin{ex}\label{bar_const_gor3}
Continuing Example \ref{codim_3_ex}, $R \otimes
\Bar A \otimes K$ is $R \otimes -$ applied to:
\[ 0 \la K_0 \xla{m_1^K} K_1 
\xla{\tiny
  \begin{pmatrix}
    m_2^K & m_1^K
  \end{pmatrix}}
\begin{array}{c}
A_1 \otimes K_0\\
K_2
\end{array}
\xla{\tiny
  \begin{pmatrix}
    m_1 \otimes 1 & 1 \otimes m_1^K & 0\\
m_2^K & m_2^K & m_1^K
  \end{pmatrix}
}
\begin{matrix}
 A_2 \otimes K_0 \\
A_1 \otimes K_1\\
K_3
\end{matrix}\]
\[\xla{ \tiny\begin{pmatrix}
  m_2 \otimes 1 & m_1 \otimes 1 & 1 \otimes m_1^K & 0\\
1 \otimes m_2^K & 0 & m_1 \otimes 1 & 1 \otimes m_1^K\\
m_3^K & m_2^K & m_2^K & m_2^K
  \end{pmatrix}}
\begin{array}{c}
 A_1^{\otimes 2} \otimes K_0 \\
A_3 \otimes K_0\\
A_2 \otimes K_1\\
A_1 \otimes K_2
\end{array}
\xla{d_5}
\begin{matrix}
A_2 \otimes A_1 \otimes K_0\\
A_1 \otimes A_2 \otimes K_0\\
A_1^{\otimes 2} \otimes K_1\\
A_3 \otimes K_1\\
A_2 \otimes K_2\\
A_1 \otimes K_3
\end{matrix}
\la \ldots
\]
with
\[ d_5 = {\small \begin{pmatrix}
 m_1 \otimes 1^{\otimes 2} & 1 \otimes m_1 \otimes 1 & 1^{\otimes 2}
  \otimes m_1^K & 0 & 0 & 0\\
m_2 \otimes 1 & m_2 \otimes 1 & 0 & 1 \otimes m_1^K & 0 & 0\\
1 \otimes m_2^K  & 0 & m_2 \otimes 1 & m_1 \otimes 1 & 1 \otimes m_1^K
& 0\\
0 & 1 \otimes m_2^K  & 1 \otimes m_2^K & 0 & m_1 \otimes 1 & 1 \otimes
m_1^K
\end{pmatrix}}.
\]
For later use, we note the maps $d_1, \ldots, d_4$ above are minimal, and $d_5$ has rank 1 after
tensoring with $k$. The non-minimality is due to the surjective
multiplication maps $A_1
\otimes A_2 \to A_3 = Q$ and $A_2 \otimes A_1 \to A_3 = Q$ ($a_i b_i = c_i$, in the notation of \ref{codim_3_ex}).
\end{ex}

\begin{rems}
There is a general construction of resolutions given in \cite[\S
7]{Burke14} that takes as further input an acyclic twisting cochain
$\tau: C \to A$. Theorem \ref{proj_resln_thm} is implicitly
using the universal twisting cochain $\tau_A: \Bar A \to
A$. If $R$ is
a codimension $c$ complete
intersection, there is an acyclic
twisting cochain $\tau: C \to A$ with $C$ the divided powers coalgebra
on a rank $c$ free $Q$-module (the
dual of $C$ is the symmetric algebra) and $A$ the Koszul complex
resolving $R$ over $Q$. The resolution of \cite[\S 7]{Burke14} applied
to this acyclic twisting cochain recovers the standard resolution for
complete intersections constructed in \cite[\S 7]{Ei80}.
\end{rems}

\section{$Q$-projective resolutions of $R$-syzygies}\label{syz_sec}
In this section we fix an \Ai-algebra
structure $m = (m_n)$ on $A$ and an \Aia-module
structure $m^G = (m_n^G)$ on $G$. Using these, we construct $Q$-projective
resolutions of all $R$-syzygies of $M$. This construction requires no
further choices once the initial \Ai-structures on $A$ and $G$ are fixed. If $Q$
is local (or graded), these
resolutions will not in general be minimal. This construction was inspired by the
``box complex'' defined in \cite[\S 8]{1306.2615}.

\begin{defn}
  The \emph{first syzygy} of $M$ over $R$ is
  \[ \syz R M := \ker( R \otimes G_0 \to M).\]
\end{defn}

This is usually defined as the kernel of any surjection from a projective $R$-module to
$M$, but since we have fixed
a surjection from a projective, we use it.

\begin{lem}\label{csyz_lem}
Consider the complex $(\p A[1] \otimes
G_0, m_1 \otimes 1)$. The degree zero map
\[\phi = s m_2^G|_{\p A[1] \otimes G_0}  : \p A[1] \otimes G_0
\to \p G[1],\]
is a morphism of complexes and there is a quasi-isomorphism
\[ \cone \phi[-2] \xra{\simeq} \syz R M,\]
thus $\cone \phi[-2]$ is a $Q$-projective resolution of $\syz R M$.
\end{lem}

\begin{proof}
 It follows from the definition of
 \Aia-module that $\phi$ is a morphism of complexes. Since $\p A[1] \otimes G_0$ is concentrated in
 non-negative homological degrees, to finish the proof it is enough to
 show that $H_i(\cone \phi[-2])
 = 0$ for $i \neq 0$ and $H_0( \cone \phi[-2]) \cong \syz R M$. To
 calculate the homology we define a map
\[ \phi': A[1] \otimes G_0 \to G[1]\]
whose restriction to $A_+[1] \otimes G_0$ is $\phi$ by setting
$\phi'( 1_A \otimes g) = g$, where $1_A$ is a basis element of $A_0 =
Q$. Note that $\cone \phi$ is
homotopic to $\cone {\phi'}$. The complex $A[1] \otimes G_0$ has
homology $R[1]
\otimes G_0$ in degree 1, is zero elsewhere, and the homology map $H(\phi')$ is a surjection $R[1] \otimes G_0 \to
M[1]$. It follows from the homology long exact sequence for the
triangle
\[ A[1] \otimes G_0 \xra{\phi'} G[1] \to \cone {\phi'} \to\]
that $H_2( \cone {\phi'} ) \cong \syz R M$ and $H_i( \cone {\phi'} ) =
0$ for $i \neq 2$.
\end{proof}

\begin{defn}\label{defn_csyz}
The \emph{first syzygy complex of $G$}, denoted $\csyz R G$, is
$\cone \phi[-2]$.
\end{defn}

\begin{rems}
The proof shows that for any map of complexes $\p A
\otimes G_0 \to \p G$ lifting the surjection $R \otimes G_0
\to M$, the cone is a $Q$-projective resolution of
$\syz R M$. But the \Ai-structure gives a canonical lift. Moreover, \ref{ai_structure_on_cone} will show that the \Ai-structures
on $A$ and $G$ determine a canonical
\Ai structure on $\csyz R G$, and thus we can iterate taking
syzygies of $G$ without making further choices.
\end{rems}

\begin{ex}\label{ex_cszy_hyp}
 Let $Q, R, A$ and $G$ be as in Example \ref{hyp_ex}. Then $\csyz R G$
 is
\[ 0 \la G_1 \xla{\tiny
  \begin{pmatrix}
    m_2^G & -m_1^G
  \end{pmatrix}}
\begin{array}{c}
  G_0 \\ G_2
\end{array}
\xla{\tiny
  \begin{pmatrix}
    0 \\ -m_1^G
  \end{pmatrix}}
G_3 \xla{-m_1^G} G_4 \la \ldots\]
\end{ex}

\begin{ex}\label{ex_csyz__gor3}
\setlength{\arraycolsep}{.1pt}
\renewcommand{\arraystretch}{1}
Let $Q, R, A,$ and $K$ be as in Example \ref{codim_3_ex}. Then $\csyz R G$
 is
\[0 \la K_1 \xla{\begin{tiny}\left (\begin{array}{cc} m_2^K &
        -m_1^K \end{array}\right )\end{tiny}} 
{\begin{matrix}
  A_1 \otimes K_0 \\ K_2
\end{matrix} }
\xla{\begin{tiny}
    \left (\begin{array}{cc}
      -m_1 \otimes 1 & 0 \\
      m_2^K & -1 \otimes m_1^K 
    \end{array} \right )
  \end{tiny}}
{\begin{matrix}
  A_2 \otimes K_0 \\ K_3
\end{matrix}} 
\xla{\begin{tiny}
    \begin{pmatrix}
      -m_1 \otimes 1 \\
      m_2^K
    \end{pmatrix}
  \end{tiny}}
A_3 \otimes K_0 
 \la 0.\]
\end{ex}

\begin{prop}\label{ai_structure_on_cone}
The maps
\[ \msyz n: \p A[1]^{\otimes n - 1} \otimes \csyz R G
\to \csyz R G\]
\[ \msyz n = 
\left (
  \begin{array}{cc}
  -s^{-1}(m_n \otimes 1)(1 \otimes s) & 0 \\
s^{-1} m_{n+1}^G (1 \otimes s) & -s^{-1}m^G_n(1 \otimes s)
  \end{array}
\right )\]
are an \Aia-module structure on $\csyz R G$.
\end{prop}

\begin{proof}
 We first construct an \Ai-structure on $\cone \phi$. Write
  \[ \cone \phi = (\p A[1] \otimes G_0)[1] \oplus \p G[1].\]
  The module $H = \p A[1] \otimes G_0$ is an \Aia-module via the maps
  $m^H_{n-1} = m_n \otimes 1$. The module $\p G[1]$ is an \Aia-module
  via the maps $-sm^G_n|_{A_+[1]^{\otimes n - 1} \otimes \p
    G}(1^{\otimes n - 1} \otimes s^{-1})$ since
  $m^G( \p G) \subseteq \p G$. Moreover, the maps
  \[ \varphi_n = s(m_{n+1}^G|_{A_+[1]^{\otimes n - 1} \otimes
    G_0}): \p A[1]^{\otimes n - 1} \otimes (\p
  A[1] \otimes G_0) \to \p G[1]\]
  form a morphism of \Aia-modules $\varphi: H \to \p G[1]$, see
  \cite[\S 3]{Burke14} for the definition, and thus the maps
  \[ m_n': \p A[1]^{\otimes n -1} \otimes \cone \phi \cong (\p
  A[1]^{\otimes n} \otimes G_0[1]) \oplus (\p A[1]^{\otimes n -1}
  \otimes \p G[1]) \]
  \[ \to (\p A[1] \otimes G_0)[1] \oplus \p G[1]\] defined as
  \begin{equation}
     m_n' = \left (
      \begin{array}{cc}
        -s(m_n \otimes 1)(1 \otimes s^{-1}) & 0 \\
        s m_{n+1}^G (1 \otimes s^{-1}) & -sm^G_n(1 \otimes s^{-1})
      \end{array}
    \right ),\label{eq:csyz_maps}
  \end{equation}
are an \Aia-module structure on $\cone {\varphi_1} = \cone {\phi}$ by \cite[5.4.4]{Burke14}. Shifting this \Aia-module twice, see
  \cite[4.3.12]{Burke14}, the maps
  \[ m^{\operatorname{syz}}_n = s^{-2} m_n' (1 \otimes s^2): \p
  A[1]^{\otimes n - 1} \otimes \cone \phi[-2] \to \cone \phi[-2]\]
are an \Aia-module structure on
  $\cone \varphi[-2] = \cone \phi[-2]$.
  Note that the shifts $s^{-2}$ and $1 \otimes s^2$ do not affect
  signs as they have degee $-2$.
\end{proof}

\begin{defn}\label{defn_of_higher_syzygy_complex}
Define $\csyz [n] R G$ inductively as $\csyz R {\csyz [n-1] R G}$,
where $\csyz [n-1] R G$ has the \Aia-module structure given by \ref{ai_structure_on_cone}.
\end{defn}

\begin{ex}
Continuing Example \ref{ex_cszy_hyp}, write $\csyz R G =  G_0[1]
\oplus \p G[-1]$. Then
\[ \msyz n = \left (
  \begin{array}{cc}
    0 & 0\\
m_{n+1}^G & -m_{n}^G
  \end{array} \right ).\]
Using this, one checks that $\csyz [2] R G$ is the box
complex of \cite[\S 7]{1306.2615}.
\end{ex}

\begin{ex}
 Continuing Example \ref{ex_csyz__gor3}, set $H = \csyz R K$. The map
\[m_2^H: A_1[1] \otimes H \to H\]
is given by
\[\small{\xymatrix@C=2.4pc@R=2pc{0 & \ar[l] A_1 \otimes K_1
    \ar[d]^{\tiny{
        \begin{bmatrix}
          0 \\ -m_2^K
        \end{bmatrix}
}}
& \ar[l] {\begin{matrix}
  A_1 \otimes A_1 \otimes K_0 \\ A_1 \otimes K_2
\end{matrix} } \ar[d]^{\tiny{
  \begin{bmatrix}
    -m_2 \otimes 1 & 0 \\
    m_3^K & -m_2^K
  \end{bmatrix}
}}
&
\ar[l]
{\begin{matrix}
  A_1 \otimes A_2 \otimes K_0 \\ A_1 \otimes K_3
\end{matrix}} \ar[d]^{\tiny{
  \begin{bmatrix}
    m_2 \otimes 1 & \, 0
  \end{bmatrix}
}}
&
\ar[l] A_1 \otimes A_3 \otimes K_0 &
\\
K_1 & \ar[l] {\begin{matrix}
  A_1 \otimes K_0 \\ K_2
\end{matrix} }
&
\ar[l]
{\begin{matrix}
  A_2 \otimes K_0 \\ K_3
\end{matrix}}
&
\ar[l] A_3 \otimes K_0 & \ar[l] 0
}}\] 
and $m_2^H: A_2[1] \otimes H \to H$ is
\[\small{\xymatrix@C=2.4pc@R=2pc{& 0 & \ar[l] A_2 \otimes K_1
    \ar[d]^{\tiny{
        \begin{bmatrix}
          0 \\ -m_2^K
        \end{bmatrix}
}}
& \ar[l] {\begin{matrix}
  A_2 \otimes A_1 \otimes K_0\\ A_2 \otimes K_2
\end{matrix} } \ar[d]^{\tiny{
  \begin{bmatrix}
    -m_2 \otimes 1 & \, 0
  \end{bmatrix}
}}
&
\ar[l]
\ar[d]^0 \ldots
\\
K_1 & \ar[l] {\begin{matrix}
  A_1 \otimes K_0 \\ K_2
\end{matrix} }
&
\ar[l]
{\begin{matrix}
  A_2 \otimes K_0 \\ K_3
\end{matrix}}
&
\ar[l] A_3 \otimes K_0 & \ar[l] 0
}}\]
Both $m_2: A_3[1] \otimes H \to H$ and $m_3^H$ are zero.
\end{ex}

The map $m_{n}^{G}$ occurs as a
component of $m_{n-1}^{\csyz R G}$. The result below generalizes this.
It implies almost
immediately the classical result of Lescot that all modules over a Golod
ring have a Golod syzygy; see \ref{syzs_over_golod_ring_are_golod}.
\begin{cor}\label{high_ai_mod_given_by_ring}
Let $G$ be a complex with
finite length $c$, and set $H := \csyz [c+1] R G$ with the \Aia-module
structure $m^H$ defined above. For $n \geq 2$,
the structure maps $m_n^H$ are defined entirely in terms of the
structure maps $m_n$ of $A$, e.g.\ the maps $m_n^G$ only factor into
$m_1^H$.
\end{cor}

\begin{proof}
Let $H_j$ be the \Aia-module $\csyz [j] R G$, defined in
\ref{defn_of_higher_syzygy_complex}. The idea of the proof is that if $k$ is the
unique integer such that $m_n^G$
occurs as a component of $m_{k}^{H_i}$, then as $j$ increases, $k$ decreases.

For $H_1 =
\csyz R G = (\p A[1] \otimes G_0) \oplus \p G[-1]$, it follows from
the formula
\[m_k^{H_1} = \pm
  \begin{bmatrix}
  -m_k \otimes 1 & 0 \\
m_{k+1}^G  &  -m^G_k
  \end{bmatrix}
\]
 that only terms $G_i$, and not $A_i
\otimes G_0$, receive a component
$m_k^G$. We show by induction this holds for all $H_j$. 
Assume that it holds for some $H_j$. We have $H_{j+1} = (\p A[1]
\otimes (H_j)_0) \oplus \p{(H_j)}$ and 
\[m_k^{H_{j+1}} = \pm \begin{bmatrix}
  -m_k \otimes 1 & 0 \\ m_{k+1}^{H_j} & -m_k^{H_j}
\end{bmatrix}.\]
By induction, the only summands in $\p {(H_j)}$ receiving $m_n^G$ for
some $n$, are $G_i$ for some $i$, and so we see this holds for $\p {(H_{j+1})}$ as well.

To finish the proof, note that if $G_i$ occurs as a summand of the complex $H_j$, then $j \geq
i$. Thus if $G_j = 0$ for $j > c$, then by the above, we see that
$m_n^G$ is not a component of $m_k^{H_{c+1}}$, for $k \geq 2$.
\end{proof}
\section{Avramov spectral sequence and obstructions}
\label{obstruction_sec}
The standard change of rings spectral sequence for the ring map $Q \to R$ transfers
homological information from $R$ to $Q$, see e.g.\ \cite[XVI, \S 5]{MR0077480}. Avramov constructed a spectral sequence in
\cite{MR601460} that transfers information from $Q$ to $R$. Iyengar gave a second construction
in
\cite{MR1442152} using the dg-bar resolution. We adapt Iyengar's arguments to construct the spectral sequence
using 
the \Ai-bar resolution, and then show how the higher homotopies on $A$ and $G$ describe the
differentials. The fact that \Ai-structures can be
used to describe differentials in Eilenberg-Moore type sequences first
appears in \cite{MR0158400part2}.

The spectral sequence depends heavily on the algebra/module structures on Tor
groups given by Cartan-Eilenberg's $\pitchfork$-product. Before we
construct the spectral sequence, we recall
the definition of the product and relate it to \Ai-structures.

For the next definition, we suspend the assumptions and notation used
previously.
\begin{defn}(\cite[\S
XI]{MR0077480}; see also \cite[Theorem 1.1]{MR1774757}.)
Let $Q$ be a commutative ring and let $R, S$ be $Q$-algebras. For
$M$ a left $R$-module and $N$ a left $S$-module, there is a homogeneous map
\[ \pitchfork: \Tor * Q R S \otimes \Tor * Q M N \to \Tor * Q M N\]
that makes $T := \Tor * Q R S$
a graded $Q$-algebra and $\Tor * Q M N$ a graded $T$-module. If $A
\xra{\simeq} R$ and $G \xra{\simeq} M$ are $Q$-projective resolutions,
let $\mu: A \otimes G \to G$ be a morphism of complexes lifting the
multiplication map $R \otimes M \to M$. Then $\pitchfork$ is defined
as \[ H_*( A \otimes S ) \otimes H_*( G \otimes N) \xra{\kappa} H_*(  A \otimes S
\otimes G \otimes N ) \cong\] 
\[H_*(  A \otimes G
\otimes S \otimes N ) \xra{ H_*(m \otimes \mu_N)} H_*( G \otimes N),\]
where $\kappa( \cls x \otimes \cls y) = \cls{ x \otimes y}$ is the
Kunneth map and $\mu_N: S \otimes N \to N$ is multiplication.\label{prod_defn}
\end{defn}

We now return to previous assumptions. There is an obvious
extension and shift of $m_2^G: \p A[1] \otimes G \to G$ to a morphism
\[ \mu: A \otimes G \to G\]
that lifts the multiplication $R \otimes M \to M$. Using
the definition of $\pitchfork$ above, we have:
\begin{lem}\label{ai_same_as_pitchfork}
Let $Q$ be a commutative ring, $R$ a cyclic $Q$-algebra, $S$ a
$Q$-algebra, $M$ an $R$-module and $N$ a left $S$-module.  Let $A$ and $G$ be $Q$-projective resolutions of $R$
and $M$, with \Ai-algebra and module structures, respectively. The map
\[ \pitchfork: \Tor * Q R S \otimes \Tor * Q M N \to \Tor * Q M N,\]
is equal to the map
\[ H_*( A \otimes S ) \otimes H_*( G \otimes N) \xra{\kappa} H_*(  A \otimes S
\otimes G \otimes N ) \cong\] 
\[H_*(  A \otimes G
\otimes S \otimes N ) \xra{ H_*(\mu\otimes \mu_P)} H_*( G \otimes N),\]
where $\mu$ is constructed from $m_2^G$ as above.
\end{lem}

We now construct the spectral sequence. Let $S$ be an $R$-algebra and
$N$ a left $S$-module. Consider the $R$-projective resolution $R \otimes \Bar A \otimes
G \xra{\simeq} M$ constructed in \ref{proj_resln_thm}. Set
\[X =  (R \otimes \Bar A
\otimes G) \otimes_R N \cong (\Bar A \otimes S) \otimes_S (G
\otimes N).\]
We filter $X$ by the number of
bars, so
\[F_p X = \bigoplus_{m \leq p} (A_+[1] \otimes S)^{\otimes m}
\otimes_S (G \otimes N).\]
This gives a first quadrant spectral sequence converging to $H_*(X) \cong
 \Tor *R M N$. Setting $\t A = A \otimes S$, we have
\[ E^0_{p,q} = \left ( \p{\t A}[1]^{\otimes p} \otimes_S (G \otimes
  N) \right)_{p+q} \cong \left ( \p{\t A}^{\otimes p} \otimes_S (G \otimes
  N) \right)_{q}.\]
If either $H_*( \t A) \cong \Tor * Q R S$ or $H_*( G \otimes N) \cong \Tor
* Q M N$ is flat over $S$, e.g.\ if $S$ is a field, then by the
Kunneth formula
\[ E^1_{p,q} \cong \left ( \Tor {+} Q R S^{\otimes p} \otimes_S \Tor * Q M N
\right )_{q}.\]
The complex 
\[ \ldots E^1_{p,*} \to E^1_{p-1, *} \to \ldots \to E^1_{1,*} \to
E^1_{0,*} \cong \Tor * Q M N \to 0,\]
using Lemma \ref{ai_same_as_pitchfork}, is isomorphic to
\[S \otimes_{\Tor * Q R S} \left ( \Tor * Q R S
\otimes_S \Bar {\left ( \Tor * Q R S \right )} \otimes_S \Tor * Q M N
\right )\]
where $\Tor * Q R S
\otimes_S \Bar {\left ( \Tor * Q R S \right )} \otimes_S \Tor * Q M N$ is
the classical bar resolution of $\Tor * Q M N$ as a graded module over
$\Tor * Q R S$. If $\Tor * Q M N$ is projective over $S$, the bar
resolution is a projective resolution, and we can make the further identification
\[ E^2_{p,q} \cong \left ( \Tor p {\Tor * Q R S} {\Tor * Q M N} S \right
)_q.\]
We have proved the following. The arguments are
adapted from \cite{MR1442152}.
\begin{thm}\label{em_ss}
Let $Q$ be a commutative ring, $R$ a cyclic $Q$-algebra, $S$ an 
$R$-algebra, $M$ an $R$-module and $N$ an $S$-module.  There is a
spectral sequence
\[E^0_{p,q} \cong \left ( \p{\t A}[1]^{\otimes p} \otimes_S (G \otimes
  N) \right)_{p+q} \cong \left ( \p{\t A}^{\otimes p} \otimes_S (G \otimes
  N) \right)_{q}\]
\[ \Rightarrow \Tor {p+q} R M N\]
with $E^0_{p,q} = 0$ for $p > q$ or $q \leq 0$. If $\Tor * Q M N$ is
projective over $S$, then
\[ E^2_{p,q} \cong \left ( \Tor p {\Tor * Q R S} {\Tor * Q M N} S \right
)_q,\]
where the algebra and module structure on $\Tor * Q R S$,
respectively $\Tor * Q M N$, is the one given in Definition \ref{prod_defn}.
\end{thm}

Now we describe the differentials using the \Ai-structures on $A$
and $G$. The general idea is that the maps $m_i, m_i^G$, for $1 \leq i
\leq r+1$, combine to form the differential on the $r$th page. 

We keep the above notation. Write the differential of $X$ restricted to $F_p X$ as $d =
t_1 + t_2 + \ldots + t_{p+1}$ with 
\[t_i: E^0_{p} \to E^0_{p-i+1} =  \sum_{j=0}^{p-i+1}
1^{\otimes j} \otimes \t m_i \otimes 1^{\otimes p - i - j +1},\]
where $\t m_i = m_i \otimes R$ for $j < p-i+1$ and $\t m_{i}
= m_{i}^G \otimes N$ for $j = p - i +1$ (the indexing is chosen so that $t_i$ involves
the maps $\t m_i$).

Following the construction of a spectral sequence from a filtration,
after e.g.\ \cite{Mac63}, set $Z^r_p = F_p X \cap d^{-1} F_{p-r} X, B^r_p = d( Z^r_{p+r}
),$ and $E^r_p = Z^r_p/(Z^{r-1}_{p-1} + B^{r-1}_p)$. An element 
\[x  = \sum_{k = 0}^p x_k \in \bigoplus_{k = 0}^p
E^0_{k} = F_p X\]
 is in $Z^r_p$ if
and only if
\begin{equation}
 \sum_{i=1}^{n} t_i( x_{p-n+i} ) = \sum_{i = 1}^n \sum_{j=0}^{p-i+1}
1^{\otimes j} \otimes \t m_i \otimes 1^{\otimes p - i - j +1}(x_{p - n
  +1}) = 0
\label{eq:cycles_in_ss}
\end{equation}
for $n = 1,
\ldots, r.$
The differential $d^r$ on $\cls y \in E^r_p$ with $y \in E^0_p$ is
given by
$d^r(\cls y) = \cls {d(x)}$ for any $x = \sum_{i = 0}^p x_i
\in Z^r_p$ with $x_p = y$. But also note that $d(x_{p-r-1} + \ldots +
x_0) \in B^{r-1}_{p-r}$ and so $\cls {d(x)} = \cls{ d( \sum_{i = p
    - r}^p x_i)}$. Finally, note that $x$ is in $Z^r_p$ if and
only if $\sum_{i = p
    - r}^p x_i$ is in $Z^r_p$.

We have proved the following. The argument is adapted from \cite[5.6]{MR0394720}. 
\begin{thm}
\label{thm_diffs_em_ss}
Let $Q$ be a commutative ring, $R = Q/I$, $S$ an $R$-algebra, $M$ an
$R$-module and $N$ an $S$-module. Let $A$ and $G $ be
$Q$-projective resolutions of $R$ and $M$, with \Ai algebra and module
structures $m$ and $m^G$, respectively. Set $\t A = A
\otimes S$, $\t m_i = m_i \otimes S$ and $\t m_i^G = m_i^G \otimes
S$. 

In the Avramov
spectral sequence
\[  E^0_{p,q} \cong \left ( \p{\t A}^{\otimes p} \otimes_S (G \otimes
  N) \right)_{q} \Rightarrow \Tor {p+q} R M N,
\]
an element $x \in
E^0_{p}$
survives to $E^r_{p}$ if and only if there exist $x_i \in E^0_{i}$ for
$i = p-r, \ldots, p - 1$ such that \eqref{eq:cycles_in_ss} holds
with $x_p = x$. In
this case $d^r(x)$ is the image of
\[ \sum_{i = 1}^{r+1} \, \, \sum_{j=0}^{p-r}
1^{\otimes j} \otimes \t m_i \otimes 1^{\otimes p - r - j - 1}(x_{p-r-1+i})
\]
in $E^r_{p-r}.$
\end{thm}

\begin{rems}
The theorem shows
that the differential on the $r$th page can be written in terms of
$m_i \otimes S$ for $i \leq r + 1$. In particular, if $m_i \otimes S =
0$ for $i \leq r+1$, then $E^0 = E^1 = \ldots = E^{r}$ in the
spectral sequence.
\end{rems}

Of particular importance, the theorem gives information on edge maps.
\begin{cor}\label{desc_of_edge_maps}
  The edge maps
  \[E^1_{0,q} = \Tor q Q M N \to \Tor {q} R M N\] are the
  change of rings map for Tor. This map takes the image of the
  multiplication
\[\pitchfork: \left (\Tor + Q R S \otimes \Tor * Q M N \right )_q \to \Tor q Q M N\]
to zero, and the induced map
\[\frac{\Tor {q} Q M N}{(\Tor + Q R S \cdot \Tor {} Q M N)_q} \to \Tor {q} R M
N\]
factors through the edge map
\[ E^2_{0,q} \to \Tor q R M N,\]
giving a commutative diagram
\[\xymatrix{ E^1_{0,q} \ar[dr] \ar@{->>}[d] & \\
\frac{\Tor {q} Q M N}{(\Tor + Q R S \cdot \Tor {} Q M N)_q}
\ar@{->>}[d] \ar[r] & \Tor q R M N.\\
E^2_{0,q} \ar[ur] &
}\]
\end{cor}

\begin{proof}
 We have $E^1_{0,q} = H_q( G \otimes N ) \cong \Tor q Q M
 N$. The edge maps on $E^1$ are induced by $G \into R
  \otimes \Bar A \otimes G$. This is a morphism of complexes from a
  $Q$-projective to $R$-projective resolution of $M$, lifting the
  identity on $M$. The change of rings map on Tor is induced by any
  such map.

Lemma \ref{ai_same_as_pitchfork} and the description of differentials
show that the multiplication
$\pitchfork: \Tor + Q R S \otimes \Tor * Q M N \to \Tor * Q M N$
factors through $E^1_{1,*} = H_*( \t A \otimes_S (G \otimes N) )
\xra{d^1} E^1_{0,*} = \Tor * Q M N$, and the result follows.
\end{proof}

We now specialize to the case $(Q, \fn, k)$ is local and Noetherian, $M$ is finitely
generated, and $S = k$ is the residue field. Recall that
\begin{defn}\label{defn_minl}
A $Q$-linear map $\phi: E \to F$ is
\emph{minimal} if $\phi(E) \subseteq \fn F$. A $Q$-free resolution
is minimal if each differential is minimal.
\end{defn}
\newcommand{\cAp}{\o {\cA}}
Every finitely generated module has a minimal free
resolution, and it is unique up to
  isomorphism. We will
  abuse language and speak of the
    minimal free resolution.

Let $A \xra{\simeq} R$ and $G
\xra{\simeq} M$ be minimal $Q$-free resolutions. Set 
\[\o A = A
\otimes k \cong \Tor {} Q R k \quad \text{ and } \quad \o G = G \otimes k \cong \Tor
{} Q M k.\]
In the spectral sequence of Theorem \ref{em_ss}, with $S= N = k$, we have
\begin{equation}\label{em_ss_N_is_k}
E^2_{p,q} = \left ( \Tor p {\Tor * Q R k} {\Tor * Q M k} k \right)_q \Rightarrow \Tor {p+q} R M k.
\end{equation}

Let $\nu_q: \Tor q Q M k \to
\Tor q R M k$ be the change of rings maps. 
By \ref{desc_of_edge_maps}, we have the following diagram
\[\xymatrix@R=5mm{
E^1_{0,q} \ar[r]^(.4){\nu_q} \ar@{->>}[d]& \Tor q R M k \\
E^2_{0,q} \ar@{->>}[d] & \\
\vdots  \ar@{->>}[d] & \\
E^{r}_{0, q} \ar[uuur]_{\alpha} &
}\]
where $\alpha$ is injective and $r = \lfloor  (q-1)/2 \rfloor$. (The
formula for $r$ comes from the fact that $E^0_{p,q} = 0$ for $q < p$.) Thus $\nu_q$ is an injection if and only
if each vertical map is an isomorphism if and only if $d^i$ is zero on
$E^i_{i, q - i +1}$ for all $2 \leq i \leq r$. From the description of the
differentials $d^i$, and the fact that $\o m_1 = 0 = \o m_1^G$, we
have the following.
\begin{cor}\label{change_of_rings_vanishing} Let $A$ be a minimal
  $Q$-free resolution of $R$ with an \Ai algebra structure
  and $G$ a minimal $Q$-free resolution of $M$ with an
  \Aia-module structure $m^G$. Let 
$$\nu_q: \Tor q Q M k \to
\Tor q R M k$$ be the change of rings map. 
  \begin{enumerate}
  \item If the maps
\[(\o m^G_n)_{q+1} = (m_n^G \otimes k)_{q+1}: \left ( \p
    A^{\otimes n - 1}[1] \otimes G \right )_{q+1} \otimes k \to G_q \otimes k\]
 are zero for all
$n$, then $\nu_q$ is injective.
\item If $\nu_q$ is injective, then $(\o m_2^G)_{q+1} = 0$.
  \end{enumerate}
\end{cor}

\begin{cor}\label{inj_change_of_rings_for_M_and_syz}
The change of rings maps
\[\nu: \Tor * Q M k \to
\Tor * R M k \quad \nu': \Tor * Q {\syz R M} k \to \Tor * R {\syz R M}k \] are injective if
and only if $\o m_n = 0 = \o m_n^G$ for all $n \geq 1$.
\end{cor}

The proof is a pleasant combination of Section \ref{syz_sec} and
Corollary \ref{change_of_rings_vanishing}.

\begin{proof}
If $\o m_n$ and $\o m_n^G$ are zero, the differentials
of the spectral sequence are zero, and so the edge morphism $\Tor * Q M k \to
\Tor * R M k$ is injective. Since $m_2^G$ is minimal, the syzygy complex $\csyz R G$
defined in \ref{defn_csyz} is minimal,
and $\o {\msyz n}$ is zero for all $n$, by definition. Thus there are no nonzero
differentials in the spectral sequence for
$\syz R M$, and so $\Tor * Q {\syz R M} k \to \Tor * R {\syz R
  M}k$ is also injective.

If $\nu$ is injective, then $\o m_2^G = 0$ by
\ref{change_of_rings_vanishing}. Thus $\csyz R G$ is a
minimal resolution of $\syz R M$ and we can apply
\ref{change_of_rings_vanishing}. This shows that if $\nu'$ is injective, then
$\o {\msyz 2} = 0$.  It follows from Definition
\ref{eq:csyz_maps} that $\o m_2 = 0$. Thus $E^1 = E^2$,
and we can repeat the argument to show that $\o m_3, \o m_3^G = 0$, in
which case $E^2 = E^3$, etc.
\end{proof}

Avramov defined 
\[o_q(M) = \ker \left ( \frac{\Tor {q} Q M k}{(\Tor + Q R k \cdot \Tor {} Q M k)_q} \to \Tor {q} R M
k \right )\] and showed that if $A$ is a dg-algebra, i.e.\ $m_n = 0$ for all
$n \geq 3$, and if
\[o(M) := \bigoplus_{q \geq 1}
o_q(M)\neq 0,\] then the minimal free $Q$-free resolution
$G$ of $M$ has no structure of a dg
$A$-module \cite[1.2]{MR601460}. We generalize this, slightly, below.

\begin{cor}
Let $(Q, \fn, k)$ be a local ring, $R$ a cyclic $Q$-algebra, and $M$ a finitely
generated $R$-module. Let $A$ and $G $ be minimal
$Q$-free resolutions of $R$ and $M$, with \Ai algebra and module
structures $m$ and $m^G$, respectively. Set
$\o m_n = m_n \otimes k$ and $\o m_n^G = m_n^G \otimes k$, and assume that $\o m_n = 0$ for $n \geq 3$, e.g.\ $A$
is a dg-algebra. 

The obstruction $o_q(M)$ vanishes if and only if the differential
$d^r: E^r_{r,q-r+1} \to E^r_{0,q}$, which is induced by
$(\o m_{r+1}^G)_{q+1}$, is zero for all $r \geq 2$. In
particular, if $\o m_r^G = 0$ for all $r \geq 3$, then $o(M) = 0.$ 
\end{cor}

The proof follows from considering the edge map
\[\beta_q: E^2_{0,q} = \frac{\Tor {q} Q M k}{(\Tor + Q R k \cdot \Tor {} Q M k)_q} \to \Tor {q} R M
k,\]
the diagram
\[\xymatrix@R=5mm{
E^2_{0,q} \ar[r]^(.4){\beta_q} \ar@{->>}[d]& \Tor q R M k \\
E^3_{0,q} \ar@{->>}[d] & \\
\vdots  \ar@{->>}[d] & \\
E^{r}_{0, q} \ar[uuur]_{\alpha} &
}\]
and the description of the differentials given in
\ref{thm_diffs_em_ss}.

This recovers Avramov's result, since if $o(M) \neq 0$, then for any
\Ai-structure $m^G$, there exists $n \geq 3$ with $\t m_n^G \neq 0$. In particular
$m_n^G \neq 0$, so by Corollary \ref{obs_cor}, $G$ does not have a dg
$A$-module structure.

\section{Golod maps}
In this section $(Q, \fn, k)$ is a local Noetherian ring, $R$ is a
cyclic $Q$-algebra and $\varphi: Q \to R$ is the projection map.

and $M$ is a
finitely generated $R$-module. We let $\varphi: Q \to R$ be projection.

\begin{defn}
An \Ai-algebra structure $m$, or \Aia-module structure $m^G$, is
minimal if
$m_n$, or $m_n^G$, is minimal for all $n \geq 2$. 
\end{defn}

We now further assume that $A \xra{\simeq} R$ and $G \xra{\simeq} M$
are minimal $Q$-free resolutions. In this section we determine when there exist minimal
\Ai-structures $A$ and $G$, or equivalently, when the bar resolution is minimal. 
We need the following notation on minimal free
resolutions.

\newcommand{\betti}[3]{\beta^{#1}_{#2}(#3)}
\begin{defn}
  The \emph{Poincare series of $M$} is the generating function
  \[ \ponser R M := \sum_{n \geq 0} \dim_k \Tor n R M k t^n.\]
This applies to any local ring, so e.g.\ $\ponser Q M$
also makes sense.
\end{defn}

\begin{lemma}\label{golod_bound_lemma}
There is a degree-wise inequality of power series
  \begin{equation}
    \ponser R M \slte \frac{\ponser Q M}{1 - t \left ( \ponser Q R - 1
      \right )}.\label{golod_bound}
  \end{equation}
This is an equality if, respectively only if, the resolution \ref{proj_resln_thm}
is minimal for some, respectively every, choice of \Ai-structures on $A, G$.
\end{lemma}

We refer to the above as the Golod bound for $M$.

\begin{proof}
The free
  resolution of \ref{proj_resln_thm} has Poincare series the rational
  function on the right. Since the minimal $R$-free resolution of $M$ is a
  subcomplex of this resolution, the inequality follows. If we can
  choose minimal \Ai-structures, then the resolution \ref{proj_resln_thm}
is clearly minimal, and so the inequality must be an equality. If the
inequality is an equality, then for any choice of \Ai-structures the
corresponding bar resolution has the same Poincare series as the minimal
resolution and so must be minimal.
\end{proof}

\begin{rems}
 The inequality \eqref{golod_bound} classical, first appearing in
 print in \cite{MR0138667} in case $M = k$ (and is credited to Serre
 there). The Eagon resolution of the
    residue field, see \cite[4.1]{MR0262227}, realizes the bound for
    this module,
    while Gokhale \cite{MR1261018} shows the existence of a resolution
    realizing the bound for any module, but he does not give an
    explicit description of the differentials. To our knowledge,
    the resolution of Theorem \ref{proj_resln_thm} is the first
    explicit realization of the Golod bound for every finitely
    generated module.
\end{rems}

\begin{defn}
 A finitely generated $R$-module $M$ is \emph{$\varphi$-Golod} if
 \eqref{golod_bound} is an equality. If $k$ is a $\varphi$-Golod
 module, then $\varphi$ is a \emph{Golod homomorphism}.
\end{defn}

\begin{rems}
Levin first defined Golod
modules \cite{LevinNotes} and Golod homomorphisms
\cite{MR0429868}. The definitions were inspired by Golod
\cite{MR0138667}. For more information, and history on Golod
    homomorphisms see \cite{MR846439} and \cite[\S 5]{MR1648664}.
\end{rems}

\begin{ex}
Let $R = Q/(f)$ with $f \in \fn^2$ a non-zero divisor. Let $K$ be
the Koszul complex on a minimal set of generators of $\fn$ over $Q$. The
construction of \cite[2.2]{MR877011} gives a homotopy $s: K \to K$
for multiplication by $f$ on $K$ such that $s$ is minimal and $(s)^2 =
0$. Following \ref{hyp_ex}, $m_2^K = s$ and $m_n^K = 0$ for $n \geq 3$
is an \Aia-module structure on $K$, and
it is clear the corresponding bar resolution of $k$ is minimal. Thus $Q \to R$ is
Golod.
\end{ex}

\begin{ex}
If is often the case that quotients of powers of ideals are
Golod. Levin shows that for a local ring $(Q, \fn)$, the map $Q \to
Q/\fn^r$ is Golod for all $r \gg 0$,
\cite[3.15]{MR0429868}, while Herzog, Welker and Yassemi show that if $(Q,
\fn)$ is regular, or a standard graded polynomial ring, then for any ideal $I$,
graded if $Q$ is the polynomial ring, the map $Q \to
Q/I^r$ is Golod for all $r \gg 0$ \cite[4.1]{1108.5862}. When $Q$ is a
standard graded polynomial ring over a field of characteristic zero,
Herzog and Huneke show that $Q \to Q/I^r$ is Golod for all graded
ideals $I$ and $r \geq 2$ \cite[Theorem 2.3]{MR3091800}.
\end{ex}

\begin{thm}\label{char_of_golod_mods}
Let $(Q, \fn, k)$ be a local Noetherian ring, $\varphi: Q \to R$ a
surjective ring map, and $M$ a finitely
generated $R$-module. Let $A \xra{\simeq} R$ and $G \xra{\simeq} M$ be
minimal $Q$-free resolutions.  The following are equivalent:
\begin{enumerate}
\item $M$ is $\varphi$-Golod;
\item the bar resolution \ref{proj_resln_thm} is minimal for some
  (respectively, every)
\Ai-structure on $A$ and $G$;
\item there exist minimal \Ai-structures on $A$ and
  $G$ (respectively, every such structure is minimal);
\item there exist $Q$-free resolutions $A'$ and $G'$ of $R$ and $Q$,
  respectively, that are not necessarily minimal, but have minimal \Ai-structures;
\item the Avramov spectral sequence \ref{em_ss}, with $S = N = k$,
  collapses on the
  first page;
\item the change of rings maps 
\[\Tor {} Q M k \to \Tor {} R M k \quad \Tor {} Q {\syz R M} k \to
\Tor {} R {\syz R M} k\]
are injective.
\end{enumerate}
\end{thm}

\begin{proof}
First note both conditions in 2 are equivalent by Lemma
\ref{golod_bound_lemma}. This also shows 1 and 2 are equivalent. Let
$m, m^G$ be \Ai-structures and assume 2 is true, so the bar resolution tensored with $k$ is zero. Using the
fact that a morphism $\begin{pmatrix}
 f_1 & f_2 \\
f_3 & f_4
\end{pmatrix}$ is zero if and only if $f_i$ is zero for all $i$, we
see that $m_n \otimes k = 0$ and $m_n^G \otimes k = 0$ for all $n
\geq 1$. Thus the \Ai-structures are minimal, and so $2$ implies the
stronger assertion of $3$. The weaker assertion of 3 tautologically
implies 4, and 4 implies 5 by the
description of the differentials of the spectral sequence in Theorem
\ref{thm_diffs_em_ss}. 
The ranks on the first page of the Avramov spectral sequence are the
same as the bound in \ref{golod_bound_lemma}, thus the sequence
collapses on the first page if and only if the bound is an equality,
i.e.\ $5$ and $1$ are equivalent. Finally,
$3$ and $6$ are equivalent by \ref{inj_change_of_rings_for_M_and_syz}.
\end{proof}

The equivalence of $1$ and $6$ was first proved by Levin in \cite{MR814185}. The
next result is also proved there. We give a short proof using
the syzygy of a resolution defined in \ref{defn_csyz}.

\begin{cor}
If $M$ is $\varphi$-Golod, then $\syz R M$ is also $\varphi$-Golod.
\end{cor}

\begin{proof}
Let $m, m^G$ be \Ai-structures. These are minimal by Theorem
\ref{char_of_golod_mods}. By \ref{ai_structure_on_cone}, $\syz R M$
has a $Q$-free resolution with minimal \Ai-structures, thus is $\varphi$-Golod.
\end{proof}

The following answers a question I learned from Eisenbud. There was no
previously bound of this type known, even in case $Q$ is regular.
\begin{cor}\label{new_bound_for_golod_mods}
Assume that $c = \max \{ \pd_Q R, \pd_Q M - 1\}$
is finite. 
If \eqref{golod_bound} is an equality up to the coefficient of
$t^{c+1}$, then $M$ is $\varphi$-Golod.
\end{cor}

\begin{proof}
If equality holds in degrees $\leq c+1$, then for any
choice of \Ai-structures the differentials $d_n$
of the bar resolution
\ref{proj_resln_thm} must be minimal for $n \leq c+2$, since the minimal free resoution is a
subcomplex of the bar resolution. As in the proof of
\ref{char_of_golod_mods}, this means that every map occurring as a
summand in degrees $\leq c+2$ is minimal. But since $A_c \otimes
G_0$ is a summand in degree $c+1$, all maps $m_n: \p A[1]^{\otimes n}
\to \p A[1]$ must have occured in degrees $\leq c + 2$ as terms $m_n
\otimes 1_{G_0}$. Also, since $G_{c+1}$ occurs in degree $c+1$, all
maps $m_n^G$ must have occurred in degrees $\leq c+2$. Thus $m_n$ and
$m_n^G$ are minimal for all $n$ and so the entire bar resolution is
minimal, and $M$ is $\varphi$-Golod.
\end{proof}

Golodness cannot be verified in fewer degrees, by the following:
\begin{ex}\label{codim_3_ex_poin_ser}
Let $Q, R$ be as in \ref{codim_3_ex}.
By \ref{bar_const_gor3}, the Poincare series of $k$ starts
\[ \ponser R k = 1 + 3t +
8t^2 + 21t^3 + 55t^4 + \ldots,\]
however
\[ \frac{\ponser Q k}{1 - t \left ( \ponser Q R - 1
      \right )} = \frac{(1+t)^3}{1 - t(5t + 5t^2 + t^3)} = 1+ 3t + 8t^2 + 21t^3 + 56t^4 + \ldots  \]
and so $k$ is not $\varphi$-Golod. Since $R$ is a codimension 3 Gorenstein ring, one could also use a
result of Wiebe, \cite[Satz 9]{MR0255531}, that shows 
\[\ponser R k = \frac{(1+t)^3}{1 - 5t^2 - 5t^3 + t^5}.\]
\end{ex}

For the
rest of the paper, we assume that
$(Q, \fn, k)$ is a regular local ring, $R \cong Q/I$ with $I \subseteq
\fn^2$, and $\varphi: Q \to R$ is projection. Any local ring that is the quotient of a regular local ring can be
written in this form; see \cite[\S 4]{MR1648664}.
\begin{defn}
A finitely generated $R$-module
is Golod if it is $\varphi$-Golod. The local ring $R$ is a Golod ring
if the residue field
$k$ is a Golod module.
\end{defn}

\begin{rems}
Whether $M$ is Golod does not depend on the presentation $R =
Q/I$. Indeed, if $R \cong Q'/I'$ with $I' \subseteq (\fn')^2$, then we
have $\ponser M Q = \ponser M
{Q'}$ for any finitely generated $R$-module $M$, and so the bound \ref{golod_bound_lemma} is an equality for $Q$
if and only if it is an equality for $Q'$. To see the equality of
Poincare series, if $K$ is the Koszul complex on a
minimal generating set of $\fn$ and $K'$ for $\fn'$, then $K \otimes
R \cong K' \otimes_{Q'} R$, since both are Koszul complexes on a
minimal generating set of the maximal ideal of $R$. So we have $\Tor * Q M
k \cong H_*( K \otimes M) \cong H_*(K' \otimes_{Q'} M) \cong \Tor *
{Q'} M k$.
\end{rems}

\begin{thm}\label{golod_ring_thm}
 Let $(Q, \fn, k)$ be a regular local ring, $I \subseteq \fn^2$ an
 ideal, and $R = Q/I$. Let $A \xra{\simeq} R$
 be the minimal free $Q$-resolution. The following are
 equivalent:
 \begin{enumerate}
 \item $R$ is Golod;
\item there exists a non-zero Golod $R$-module;
\item there exists a minimal \Ai-structure on $A$ (respectively, every
  \Ai-structure on $A$ is minimal);
\item the change of rings map for the maximal ideal $\fm$ of $R$,
\[ \Tor {} Q {\fm} k \to \Tor {} R {\fm} k,\]
is injective;
\item the inequality \eqref{golod_bound}, for the module $M = k$, is an equality up to the coefficient of
$t^{e+1}$, where $e = \dim Q$.
 \end{enumerate}
\end{thm}

\begin{proof}
Set $K$ to be the Koszul complex on $\fn$ with
an \Aia-module structure $m^K$.  Since $I \subseteq
\fn^2$, the change of rings map $\nu: \Tor {} Q k k \to \Tor {} R k k$ is
injective. This is well known. To see it, one can minimally resolve $k$ over
$R$ by adjoining higher degree dg-variables to the dg $R$-algebra $K \otimes R$, as
in \cite[\S 6]{MR1648664}. 

For the implications, we have $1 \Rightarrow 2$ by definition, and $2 \Rightarrow 3$ by
\ref{char_of_golod_mods}. If $3$ holds, then $m \otimes k =
0$, and arguing as in the proof of
\ref{inj_change_of_rings_for_M_and_syz},  the injectivity of $\nu$
implies that $m^K \otimes k = 0$. Thus by \ref{change_of_rings_vanishing}
applied to $M = k$, 4 holds. Since $\nu$ is always injective, we have
$4 \Rightarrow 1$ by \ref{inj_change_of_rings_for_M_and_syz} and
\ref{char_of_golod_mods}. Finally, $1 \iff 5$ holds by
\ref{new_bound_for_golod_mods}.
\end{proof}

The equivalence of 1 and 4 was first proved in \cite{MR0429868},
and that of 1 and 2 in \cite{MR814185} and \cite{MR1060830}, independently.

\begin{rems}
Golod rings are also characterized as having a homotopy Lie algebra
that is free in degrees $\geq 2$, or as having trivial Massey
products on the homology of the Koszul complex, or those for which
Eagon's resolution of the residue field is minimal; see
\cite{MR846439, MR0262227}. There seems to be interesting, but delicate, connections between these three
characterizations and \Ai-structures. We plan to return to this in
future work.
\end{rems}

The following was first proved by Lescot in
\cite{MR1060830}. We give
a short proof using the syzygy construction of a free resolution.
\begin{cor}\label{syzs_over_golod_ring_are_golod}
Let $R$ be a Golod ring, $M$ a finitely generated $R$-module, and let
$c = \pd_Q M$. Then $\syz [c+1] R M$ is a Golod module.
\end{cor}

\begin{proof}
By the above theorem, $A$ has a minimal \Ai-structure. Thus by \ref{high_ai_mod_given_by_ring} there is a
(possibly non-minimal)
$Q$-free resolution of $\syz [c+1] R M$ with a minimal \Aia-module structure, and so
by \ref{char_of_golod_mods}.(4), it is a Golod module.
\end{proof}

By Theorem \ref{golod_ring_thm}, to show $R$ is Golod it is enough to show that $\o m_n = 0$ for all
 $n$. This is connected to a folklore question. There is an isomorphism of
graded $k$-algebras $H_*( K
\otimes R ) \cong \o A$, where $K$ is the Koszul complex on $\fn$, the
maximal ideal of $Q$. There is no known
example of a non-Golod ring $R$ with trivial multiplication on $H_*(K
\otimes R)$. The question is whether such non-Golod rings
exist. In terms of \Ai-algebras, this is the following.

\begin{quest}
Let $R = Q/I$ be a local ring, with $(Q, \fn)$ a
regular local ring and $I \subseteq \fn^2$. Let $A$ be the minimal
$Q$-free resolution of $R$ with \Ai-algebra structure $(m_n)$. If $\o m_2 = 0$, is $\o m_n = 0$ for all $n \geq 2$?
\end{quest}

By \cite{MR2344344}, this is true when $Q$ is a graded polynomial ring
and $I$ a monomial ideal.

\begin{ex}
Theorem \ref{golod_ring_thm} gives an easy proof of a result of
Shamash \cite{MR0241411}. Let $(Q, \fn)$ be a regular local ring, $J$ be an ideal in $Q$, and $f \in \fn$ a nonzero element. Then
$R = Q/(f\cdot J)$ is
 Golod. Indeed, let $B \xra{\simeq} Q/J$
be the minimal $Q$-free resolution. Since $f$ is
 a non-zero divisor, multiplication by $f$ gives an isomorphism of
 $Q$-modules 
\[ J \xra{ \cong} f \cdot J.\] 
Thus we can construct a minimal $Q$-free resolution of
 $R$ by setting $\p A = \p B$, $A_0 = Q$ and $d_1^A =
 f \cdot d_1^B$. If $m^B_n: \p
 B[1]^{\otimes n} \to \p B[1]$ is an
 \Ai-algebra structure on $B$, then
\[m_n^A = f^{n-1}
 m^B_n: \p A[1]^{\otimes n} \to \p A[1]\] is an \Ai-algebra structure on $A$. In particular,
 $m_n^A$ is minimal for all $n$, so $R$ is Golod.
\end{ex}

\begin{cor}
If $R$ is a Golod local ring, then one can construct the minimal $R$-free resolution
of every finitely generated module in finitely many steps.
\end{cor}

\begin{proof}
Let $M$ be a finitely generated $R$-module. By \ref{syzs_over_golod_ring_are_golod}, a syzygy
$N$ of $M$ is a Golod module. We
can then construct the finite minimal $Q$-free resolutions of $R$ and $N$, and
the finitely many maps in \Ai-structures on these. By
\ref{char_of_golod_mods}, 
the resolution \ref{proj_resln_thm} using these \Ai-structures is the minimal resolution of
$N$.
\end{proof}
\def\cprime{$'$}

\end{document}